\documentclass[final,8pt, 1p,times]{elsarticle}

\usepackage{mathptmx}
\usepackage{graphicx}

\usepackage{amsfonts}
\usepackage{ntheorem}

\usepackage{amsmath}

 \newtheorem{theorem}{Theorem}
\newtheorem*{theorem-non}{Theorem}
\newtheorem{lemma}[theorem]{Lemma}
 \newtheorem{proposition}[theorem]{Proposition}
 \newtheorem{corollary}[theorem]{Corollary}
\renewtheorem{lemma}[theorem]{Lemma}
 
\newproof{pf}{Proof}

\usepackage{algorithm}

\begin{document}

\begin{frontmatter}

\title{A balanced k-means algorithm for weighted point sets{\small\tnoteref{t1}}}

\author[TUM,UCD]{S. Borgwardt\corref{cor1}}
\ead{steffen.borgwardt@ucdenver.edu}

\author[UniBW]{A. Brieden}
\ead{andreas.brieden@unibw.de}

\author[TUM]{P. Gritzmann}
\ead{gritzmann@tum.de}

 \cortext[cor1]{Corresponding author.}

 \tnotetext[t1]{Research of the authors was supported in part by the EURO Excellence in Practice Award 2013.}

\address[TUM]{Fakult\"at f\"ur Mathematik, Technische~Universit\"at M\"{u}nchen, Garching-Forschungszentrum, Germany}

\address[UniBW]{Fakult\"at f\"ur Wirtschafts- und Organisationswissenschaften, Universit\"at der Bundeswehr M\"{u}nchen, Germany}

\address[UCD]{Department of Mathematical and Statistical Sciences,
  University of Colorado Denver, USA}


\begin{abstract}
The classical $k$-means algorithm for partitioning $n$ points 
in $\mathbb{R}^d$ into $k$ clusters is one of the most 
popular and widely spread clustering methods. The need to respect prescribed lower bounds on the cluster sizes has been observed in many scientific and 
business applications. 

In this paper, we present and analyze a generalization of $k$-means that is capable of handling weighted point sets and prescribed lower and upper bounds on the cluster sizes. We call it {\em weight-balanced $k$-means}. The key difference to existing models lies in the ability to handle the combination of weighted point sets with prescribed bounds on the cluster sizes. This imposes the need to perform partial membership clustering, and leads to significant differences.

For example, while finite termination of all $k$-means variants for unweighted point sets is a simple consequence of the existence of only finitely many partitions of a given set of points, the situation is more involved for weighted point sets, as there are infinitely many partial membership clusterings. Using  polyhedral theory, we show that the number of iterations of weight-balanced $k$-means is bounded above by  $n^{O(dk)}$, so in particular it is polynomial for fixed $k$ and $d$. This is similar to the known worst-case upper bound for classical $k$-means for unweighted point sets and unrestricted cluster sizes, despite the much more general framework. We conclude with the discussion of some additional favorable properties of our method.

\end{abstract}


\end{frontmatter}



\thispagestyle{plain}

\section{Introduction}

 The well-known {\bf $k$-means algorithm} (or Lloyd's
algorithm) \cite{m-67,l-06} is one of the most influential and popular clustering methods. Its use reaches into many fields of analyzing data for rational decision-making, with an abundance of applications in operations research ranging from facility location to risk prediction.

 Given a set of $n$ points and $k$
different initial sites in $\mathbb{R}^d$, 
it iteratively performs two steps (ignoring some technical details), see Algorithm \ref{algo:bklsa}: 
First, every point is assigned to a closest site. 
This partitions the points into $k$ clusters, one for each site. Second,
the sites are updated to be the arithmetic means of the
clusters. These two-step iterations are performed until the sites do
not change anymore. The algorithm exhibits many favorable properties,
and is well-accepted for its simplicity and fast convergence in
practice. Also, it is still subject to quite extensive theoretical
analysis motivated by the discrepancy between its excellent behaviour in 
practice and its known worst-case behaviour; see \cite{amr-11,vattani-11}.

The need to respect given lower bounds on the cluster sizes has been observed in many applications, particularly for guaranteeing that no empty clusters are created \cite{bbd-00,dbb-08,mf-14}. The desire to create clusters of prescribed sizes also arises when data segmentation is applied to arrive at more homogenous clusters for further statistical analysis. See \cite{bdw-09} for an introduction into general constrained clustering and more applications and \cite{hr-12,hrc-13} for some background on clustering with respect to prescribed cluster sizes. 

In the present paper, we present a generalization by introducing 
{\bf weight-balanced $k$-means}. Within this framework, we are able to
deal with weighted point sets and prescribed upper and
lower bounds for the cluster sizes. 

The combination of weighted point sets and the need to create clusters of prescribed cluster sizes also arise in many applications, particularly in data analytics, risk prediction, and predictive maintenance. Note that, in particular, weighted points allow for a natural 
representation of identical, repeated points in data sets. A recent application that already used a  preliminary version of the method in the present paper investigates the semantic structure behind SAR image collections \cite{bbrd-14}. In this approach, weights for a point set are assigned through user interaction. A preliminary version of the method in this paper was also applied in material science for the
representation of polycrystals \cite{abglp-15}.

 Another problem with both a weighted pointed set and the need to respect given cluster sizes arises in the consolidation of farmland \cite{bg-03,bbg-09,bbg-14}: In an
agricultural region, $k$ farmers cultivate $n$ lots. The goal is to
improve the cost-effective structure in this region by a combinatorial
redistribution of these lots, a {\em voluntary lend-lease agreement} 
after which each farmer can work on larger adjacent components.
The lots differ in value, and no farmer would accept a considerable 
deviation from his original total value during this redistribution. 
By representing the lots by points in the Euclidean plane and 
using their values as weights, we arrive at a weighted clustering problem 
with relatively tight bounds on the cluster value for each farmer.

In contrast to partitioning an unweighted point set, as facilitated
by $k$-means, a weighted point set imposes the need to perform partial
membership clustering, where points can be fractionally assigned to
more than one cluster. This is unavoidable in general if, e.g., a
data set consists of two points of weight three and has to be
clustered into three parts of weight two or, if a set of three points 
of equal weight has to be partitioned into two clusters of
equal size. As it turns out, the number of fractionally assigned weighted points can be strictly limited, see e.g. \cite{bbg-09}.

We set up the generalized $k$-means framework by modelling the computation of a so-called
{\em weight-balanced least-squares assignment} as a special linear program.
Its feasible region, the {\em weight-balanced partition polytope}, 
encodes specific cell decompositions of $\mathbb{R}^d$. Such {\em power diagrams}
are generalized Voronoi diagrams (see \cite{ak-99} for a survey), 
and appear in learning as the classifiers derived by alltogether models
for multiclass classification; see e.g., \cite{v-98,ww-98,bb-99,cs-01}.
Informally, a classifier is a rule for determining to which of the existing 
clusters a new point in $\mathbb{R}^d$ should be assigned. In our case, each 
cluster is associated with exactly one of the cells, and a new point is 
assigned to the clusters of (one, some, all) cells it lies in. 

We use these intimate connections to guarantee termination of our
algorithm within $n^{O(dk)}$ iterations. In order to place
this result into perspective note that by \cite{iki-94} there is 
a similar worst-case bound for classical $k$-means. Hence we obtain our 
generalization at essentially no additional cost. Further, the upper bound shows
that the algorithms runs in polynomial time for fixed $k$ and $d$. While classical $k$-means has a polynomial smoothed running time, 
\cite{amr-09,amr-11}, it may in general require exponentially many 
iterations even in the plane, \cite{vattani-11}.

In any case, results on the running time of $k$-means variants should be contrasted with the $\mathbb{NP}$-hardness of 
the {\em $k$-means problem}, the problem of finding globally optimal sites, for $k=2$, \cite{adhp-09}, or $d=2$, \cite{mnv-12}. There even is an $\epsilon >0$ for which it is $\mathbb{NP}$-hard to find a $(1+\epsilon)$-approximation for the problem \cite{acks-15,lsw-15}. There are, however, polynomial-time approximation algorithms even in the weight-balanced case with worst-case error bound that depends only on $d$ and $k$. \cite{bg-10} gives a more general analysis within the framework of norm-maximization. As far as we know, the complexity of computing any local optimum
of the $k$-means problem is still open.


The present paper is organized as follows. Section \ref{sec-preliminaries}
will provide the relevant terminology and concepts. In Section \ref{sec-results},
we outline our main results. Section \ref{sec-wbls-clustering} then prepares tools involving least-squares assignments and power diagrams.  In
Section \ref{sec-weight-balanced-k-means}, we present our generalized $k$-means algorithm, and
prove its correctness and termination. In particular, we derive the indicated bound
on the number of iterations of our algorithm. We conclude the paper in Section \ref{sec-final} with some remarks on additional favorable properties of our algorithm.

\section{Preliminaries}\label{sec-preliminaries}

We begin with some standard basic notation.

\subsection{Weight-balanced clusterings}
Let throughout the present paper $k,n,d \in \mathbb{N}$ with $n\geq
k\geq 2$. Let $X:=\{x_1,\dots,x_n\}\subset \mathbb{R}^d$ be a {\bf
  data set} of distinct points with associated {\bf weights}
$\Omega:=(\omega_1,\dots,\omega_n)\in \mathbb{R}^n$ and $\omega_i>0$
for all $i\leq n$. Without loss of generality, we assume that $X$ is
full-dimensional, i.e., the dimension of its affine hull is $d$; in
particular then $d+1\leq n$. Otherwise we could work in the affine
hull of $X$. Further, let $\kappa^-:=(\kappa_1^-,\dots,\kappa_k^-)^T,
\kappa^+:=(\kappa_1^+,\dots,\kappa_k^+)^T\in \mathbb{R}^k$ with
$0<\kappa_i^-\leq \kappa_i^+$ and 
$$
\sum\limits_{i=1}^k \kappa_i^-
\leq \sum\limits_{j=1}^n \omega_j \leq \sum\limits_{i=1}^k
\kappa_i^+\;\;.
$$

A {\bf (partial membership) $k$-clustering} $C:=(C_1,\dots,C_k)$ of $X$
consists of $k$ clusters $C_i$, and is defined by an assignment vector
$y:=(y_{11},\dots,y_{1n},\dots,y_{k1},\dots,y_{kn})^T\in [0,1]^{kn}$
of $X$ with $\sum\limits_{i=1}^k y_{ij}=1$ for all $j\leq n$. 
Informally, $y_{ij}$ is the fraction of the total weight
$\omega_j$ of point $x_j$ that belongs to $C_i$. Formally, we set
$C_i:= (y_{i1},\dots,y_{in})^T$. The {\bf support of the cluster $C_i$} is 
$$
\text{supp}(C_i):=\{x_j: y_{ij}>0\},
$$ 
the {\bf support of the clustering $C$} is the tuple
$\text{supp(C)}:=(\text{supp}(C_1),\dots,\text{supp}(C_k))$.

We use the notation $|C_i|:=\sum\limits_{j=1}^n y_{ij}\omega_j$ to
refer to the total weight or the {\bf size} of cluster $C_i$. The
tuple $|C|:=(|C_1|,\dots,|C_k|)^T$ is the {\bf shape} of $C$. The {\bf center
of gravity} $c_i$ of $C_i$ is given
by $$c_i:=\frac{1}{\sum\limits_{j=1}^n y_{ij} \omega_j}
\sum\limits_{j=1}^n y_{ij} \omega_j x_j\;\;.$$ In this paper we deal
with the task of finding {\bf weight-balanced clusterings}
that satisfy $\kappa^-\leq|C|\leq \kappa^+$ componentwisely.  

We will call a tuple $I=(k,n,d,X,\Omega,\kappa^\pm)$ that satisfies all the above properties an {\bf instance} for the weight-balanced clustering problems at hand. Note that for an instance $I$, there always exists a weight-balanced clustering. A tuple $I'=(k,n,d,X)$, without weights and without cluster size restrictions, is referred to as a {\bf trivial instance}.

\subsection{Weight-balanced least-squares assignments}\label{subsec-lc-assignments}

Least-squares assignments are a common kind of clustering, and appear
in many real-world applications like facility location or clustering
with low variance. They allow intuitive interpretations such as
measuring the cost to supply customers with geographic positions
$x_1,\dots,x_n$ from supply sites $s_1,\dots,s_k\in \mathbb{R}^d$ with
respect to a quadratic-loss transport. Several classical clustering
algorithms are based on least-squares assignments, a prime example
being the $k$-means algorithm, which computes such an assignment in
each iteration.

Let $S:=\{s_1,\dots,s_k\}$ be a set of {\bf sites } in
$\mathbb{R}^d$. A clustering $C:=(C_1,\dots,C_k)$ is a weighted
$S$-least-squares assignment of $X$ if and only if
\[
\sum\limits_{i=1}^k \sum\limits_{j=1}^n y_{ij}\omega_j\cdot
\|x_j-s_i\|^2
\]
is minimal for all clusterings of $X$ for the given sites $S$.  We call $C$ a 
{\bf weight-balanced $(S,\kappa^-,\kappa^+)$-least-squares assignment} if
it is weight-balanced and minimal with respect to all clusterings of $X$ of shape $
\kappa^-\leq|C|\leq \kappa^+$. We will add {\bf `strict'} to signify
that the minimal assignment is unique.

In the degenerate case with $s_i=s_j$ for some $i\neq j$, we do not
distinguish between clusters $C_i$ and $C_j$. Rather, we
treat them as a single cluster $C_{ij}$ defined by $C_{ij}:=
(y_{i1}+y_{j1},\dots,y_{in}+y_{jn})^T$. The cluster adheres to the
size bounds $\kappa_{ij}^{-}:=\kappa_{i}^-+\kappa_j^-$ and
$\kappa_{ij}^{+}:=\kappa_{i}^++\kappa_j^+$. The strict minimality of
the underlying clustering refers to a comparison with all clusterings
where $C_i$ and $C_j$ are treated as the combined cluster $C_{ij}$. In the following, we will therefore assume that all $s_i$ are different.

\subsection{Feasible power diagrams}\label{subsec-power-diagrams}

Power diagrams are special kinds of cell decompositions, and
generalize the well-known Voronoi diagrams. In machine learning, they arise as
classifiers derived by alltogether models for multiclass
classification. See \cite{a-87} for a survey of this data structure.

A power diagram is specified by a set of distinct sites
$S:=\{s_1,\dots,s_k\}\subset \mathbb{R}^d$ and parameters
$\Sigma:=(\sigma_1,\dots,\sigma_k)\in \mathbb{R}$. Using these
parameters, the {\bf $i$-th power cell} $P_i$ is defined by
$$
P_i:=\{x\in\mathbb{R}^d: \| x- s_i\|^2 - \sigma_i \leq  \| x- s_j\|^2 - \sigma_j 
\text{ for all } i \neq j\}.
$$
Here $\| \cdot \|$ denotes the Euclidean norm. The {\bf power diagram} $P$
then is the tuple $P:=(P_1,\dots,P_k)$.

For our weighted point sets, we say that a power diagram $P$ is 
{\bf feasible} for the clustering $C$ if $\text{supp}(C_i) \subset P_i$
for all $i \leq k$. Then {\bf $C$ allows $P$}. 
More strongly,  {\bf $P$ supports $C$} if $\text{supp}(C_i)=P_i\cap X$  for all $i\leq k$. 

In this paper, a special kind of feasible power diagrams, the
so-called {\bf strongly feasible power diagram $P$} is important,
\cite{bg-11}, that supports $C$ and has the following additional property.
Let $G(C)$ be the multigraph with vertices $C_1,\dots,C_k$ and an edge 
labeled with $x_j$ incident to $C_i$ and $C_l$ with $i\ne l$ if and only if 
$x_j\in \text{supp}(C_i) \cap \text{supp}(C_l)$. Then $G(C)$ does not 
contain a cycle with two or more different edge labels.

If $P$ is a feasible power diagram for $C$, and the sites coincide
with the centers of gravity of the clusters, i.e., $s_i=c_i$ for all
$i\leq k$, then $P$ is called a {\bf centroidal power diagram}.

\section{Main results}\label{sec-results}

We present a generalization of the $k$-means algorithm that can handle the combination of 
weighted point sets and prescribed lower and upper bounds on the
cluster sizes. More precisely, Algorithm \ref{algo:wbkm} accepts
weighted points, sites and lower and upper bounds for the cluster
sizes and computes a strongly feasible centroidal weight-balanced
least-squares assignment.

While the termination of the classical $k$-means algorithm is clear
(when appropriate precautions are taken if empty cells occur)
because none of the finitely many clusterings is visited twice,
finiteness of Algorithm \ref{algo:wbkm} is not obvious. In fact, since
we deal with partial membership clusterings, there is an infinite
number of possible states. However, our first result proves
termination.

\begin{theorem-non}
  Algorithm \ref{algo:wbkm} terminates with a clustering that allows a
  strongly feasible centroidal power diagram.
\end{theorem-non}

Inaba et al. \cite{iki-94} proved an upper bound of $n^{O(kd)}$ for
the worst-case running time of the classical $k$-means algorithm, by
bounding the total number of least-squares assignments of a data
set. See also \cite{hor-99,os-01} for other bounds on the number of
related types of clusterings. We derive a similar upper bound in our more general framework.  More percisely, with $e$ denoting the
Euler number we obtain the following bound on the number of iterations
of our algorithm which is polynomial for fixed dimension and fixed number of
clusters.

\begin{theorem-non}
  The number of iterations of Algorithm \ref{algo:wbkm} is bounded by
  $(40ek^2n)^{(d+1)k-1}.$
\end{theorem-non}

The proofs of these theorems are given in Section \ref{subsec-termination} (as
Theorems \ref{thm:termination} and \ref{thm:numberofiterations}).

\section{Weight-balanced least-squares assignments}\label{sec-wbls-clustering}

In this section, we recall how to compute a weight-balanced
least-squares assignment for a given instance $I=(k,n,d,X,\Omega)$ and a set of sites $S$. This computation has been studied in \cite{bg-11}. Proposition \ref{prop:LSAPD} will list the results we need for our discussion. To make the paper self-contained, we give the detailed model in our notation.

\subsection{A linear programming formulation}\label{subsec-lp-formulation}
We are interested in some weight-balanced $(S, \kappa^-, \kappa^+)$-least-squares assignment of $X$. We can
describe the task to find such an assignment as an optimization
problem in the form
$$ 
\min\, \sum\limits_{i=1}^k \sum\limits_{j=1}^n y_{ij}\omega_j \cdot \| x_j-s_i \|^2 
\quad\text{ s.t. }\quad \kappa^- \leq|C_i|\leq \kappa^+ \quad (i\leq k)\;\;.  
$$
We use the variables $y_{ij}\in [0,1]$ to indicate how much of the weight
$\omega_j$ of point $x_j$ is associated to cluster $C_i$. Then the
constraints can be described by the
set
$$
\begin{array}{lcrclcl}
  \kappa_i^- & \leq   & \sum\limits_{j=1}^{n}  y_{ij}\omega_j    & \leq & \kappa_i^+ & \quad & (i\leq k)\hfill \;\\
  &      & \sum\limits_{i=1}^k y_{ij} & =  &  1
  & \quad &  (j\leq n)\hfill \;\\
  &      &          y_{ij}     & \geq   & 0                   &  & (i \leq k,
  j \leq n)\;\;
\end{array}
$$
of linear equalities. The first line in this system implies that the
clusters satisfy the given bounds, whereas the second line guarantees
that each point is fully assigned. Note that these constraints define
a polytope; we call it the {\bf weight-balanced partition polytope}.

The objective function can be written as
$$
\min\, \sum\limits_{i=1}^k \sum\limits_{j=1}^n  y_{ij}\omega_j \cdot\| x_j-s_i \|^2
=\min\, \sum\limits_{i=1}^k \sum\limits_{j=1}^{n} y_{ij}\omega_j \cdot  (x_j^Tx_j-2x_j^Ts_i+s_i^Ts_i)\;\;,
$$
which is linear in the $y_{ij}$, as $x_j, s_i$ and $\omega_j$ are
fixed. Note further that $$\sum\limits_{i=1}^k \sum\limits_{j=1}^{n}
y_{ij}\omega_j\cdot (x_j^Tx_j-2x_j^Ts_i+s_i^Ts_i) =
\sum\limits_{j=1}^{n} \omega_j \cdot x_j^Tx_j + \sum\limits_{i=1}^k
\sum\limits_{j=1}^{n} y_{ij}\omega_j\cdot (s_i^Ts_i-2x_j^Ts_i)\;\; .$$
Thus, the objective function can be simplified to
$$
\min\, \sum\limits_{i=1}^k \sum\limits_{j=1}^{n}  y_{ij}\omega_j\cdot  (s_i^Ts_i-2x_j^Ts_i)\;\;.
$$
We formally sum up this construction in Algorithm \ref{algo:blsa} and the subsequent lemma. 

\begin{algorithm}
  \begin{itemize}
  \item {\bf Input}: Instance $I=(k,n,d,X,\Omega,\kappa^\pm)$, $S:=\{s_1,\dots,s_k\}\subset\mathbb{R}^d$
  \item {\bf Output:} A $(S,\kappa^-,\kappa^+)$-least-squares
    assignment of $X$
  \end{itemize}
  Solve the linear program

$$
\begin{array}{lcrclcl}
  \multicolumn{5}{c}{\min\, \sum\limits_{i=1}^k \sum\limits_{j=1}^{n} y_{ij}\omega_j\cdot   (s_i^Ts_i-2x_j^Ts_i)}  &       &\; \\
  \kappa_i^- & \leq   & \sum\limits_{j=1}^{n}  y_{ij}\omega_j    & \leq & \kappa_i^+ & \quad & (i\leq k)\hfill \;\\
  &      & \sum\limits_{i=1}^k y_{ij} & =  &  1
  & \quad & (j\leq n)\hfill \;\\
  &      &          y_{ij}
  & \geq   & 0                   &  & \hfill (i \leq k,
  j \leq n)\;\;.
\end{array}
$$
and return a basic feasible solution as the assignment.
\caption{Weight-balanced least-squares assignment}
\label{algo:blsa}
\end{algorithm}

\begin{lemma}\label{lem:blsa}
  Algorithm \ref{algo:blsa} computes a weight-balanced least-squares
  assignment by linear programming with $k\cdot n$ variables and
  $(k+1)\cdot n +2\cdot k$ constraints.
\end{lemma}

\subsection{Strongly feasible power diagrams}\label{subsec-strongly-feasible-pd}
When varying the sites, the different outputs of Algorithm \ref{algo:blsa} 
constitute a special subclass $\mathcal{V}$ of the vertices of the
weight-balanced partition polytope. For our purposes we need a characterization of
$\mathcal{V}$ in terms of strongly feasible power diagrams.

Feasible power diagrams are intimately related to least-squares
assignments of unweighted point sets. Aurenhammer et al. \cite{aha-98}
proved that if $C$ is a balanced least-squares assignment of an
unweighted data set (i.e., $\Omega=(1,\dots,1)$) and
$\kappa^-=\kappa^+$, then $C$ allows a strongly feasible power
diagram. It is also possible to obtain this statement by a careful
interpretation of Theorem $5$ in \cite{bhr-92}. These results allow for a
far-reaching extension, \cite{bg-11}, leading, in particular, to a complete
characterization in the case $\kappa^-=\kappa^+$. Here we only need
the following particular result (see \cite{bg-11}, Corollaries 2.2, 2.3 
(including their proofs), and Lemmata 4.2, 4.3).

 \begin{proposition}[\cite{bg-11}]\label{prop:LSAPD}
   Let $X$ be a weighted data set, and let $C$ be a (strict)
   weight-balanced least-squares assignment for $X$. Then $C$ allows a
   (strongly) feasible power diagram. 

Further, if $C$ allows a strongly feasible power diagram, its assignment vector $y$ contains
   at most $2(k-1)$ fractional components, i.e., components with
   $0<y_{ij}<1$.
 \end{proposition}

 In particular, if $C$ corresponds to a vertex in $\mathcal{V}$ then, by Proposition \ref{prop:LSAPD},
 $C$ allows a strongly feasible power diagram.

\begin{corollary}\label{cor:WeComputeSPD}
  Algorithm \ref{algo:blsa} computes a clustering $C$ that allows a
  strongly feasible power diagram.
\end{corollary}

  By Lemma \ref{lem:blsa}, Algorithm \ref{algo:blsa} computes a
  weight-balanced least-squares assignment $y$ and returns a vertex of the
  weight-balanced partition polytope. By Proposition \ref{prop:LSAPD},
  it corresponds to a clustering $C$ that allows a strongly feasible
  power diagram. \hfill $\Box$

\section{Weight-balanced $k$-means}\label{sec-weight-balanced-k-means}

The classical $k$-means algorithm is one of the most widely used
clustering algorithms. As a service to the reader, we state a basic
version as Algorithm \ref{algo:bklsa}.

\begin{algorithm}[h]
  \begin{itemize}
  \item  {\bf Input}: Trivial instance $I'=(k,n,d,X)$, $S:=\{s_1,\dots,s_k\}\subset\mathbb{R}^d$
 \item  {\bf Output}: A least-squares assignment of $X$ for the
  arithmetic means as sites
\end{itemize}
  \begin{enumerate}
  \item Partition $X$ into a clustering $C=(C_1,\dots,C_k)$ by
    assigning $x_j \in X$ to the cluster $C_i$ of a closest site
    $s_i\in S$.
  \item Update each site $s_i$ as the center of gravity of cluster
    $C_i$; if $|C_i|=0$, choose $s_i=x_l$ for a random $l\leq n$ with
    $x_l\neq s_j$ for all $j\leq k$. If the sites change, go to
    $(1.)$; else
    return the current assignment and sites.
  \end{enumerate}
  \caption{$k$-means}\label{algo:bklsa}
\end{algorithm}

\subsection{Balanced $k$-means for weighted point sets}\label{subsec-balanced}
By replacing the trivial assignment step $(1.)$ with the computation of
a strongly feasible weight-balanced clustering according to Algorithm
\ref{algo:blsa}, we can generalize this framework to deal with
weighted point sets and lower and upper bounds on the cluster
sizes. Algorithm \ref{algo:wbkm} describes this in pseudocode.

\begin{algorithm}
  \begin{itemize}
  \item {\bf Input}: Instance $I=(k,n,d,X,\Omega,\kappa^\pm)$, 
    $S:=\{s_1,\dots,s_k\}\subset\mathbb{R}^d$
  \item {\bf Output}: A strict weight-balanced
    least-squares assignment of $X$ for its centers of gravity as
    sites
  \end{itemize}
  \begin{enumerate}
  \item Apply Algorithm \ref{algo:blsa} for the current set of sites
    to obtain assignment $y$.
  \item Update each site $s_i$ as the center of gravity
    $c_i:=\frac{1}{\sum\limits_{j=1}^n y_{ij}\omega_j}
    \sum\limits_{j=1}^n y_{ij}\omega_j x_j$. 

If the objective function value decreased during the last iteration, go to $(1.)$; else
    return the current assignment and sites.
  \end{enumerate}
  \caption{weight-balanced $k$-means}
  \label{algo:wbkm}
\end{algorithm}

\subsection{Correctness and termination}\label{subsec-termination}
It is easy to show correctness and termination of the standard
$k$-means algorithm. The trivial assignment of each point to a closest
site in each iteration is readily interpreted as the computation of an
unconstrained least-squares assignment. A straightforward
argument then shows that the center of gravity of a fixed (non-empty)
cluster is the site for which this least-squares distance is minimal:
$$
c_i= \frac{1}{|C_i|} \sum\limits_{x \in C_i} x 
= \arg \min\limits_{c\in \mathbb{R}^d} \sum\limits_{x\in C_i} \|x-c\|^2.
$$
The sum of these sums of squared distances forms a strictly
decreasing sequence over the course of the algorithm. As there is only
a finite number of different clusterings of the finite point set $X$,
the algorithm terminates.

For our scenario with weighted points and partial
membership clustering, a proof of correctness is slightly more involved: There
is an infinite number of such clusterings, so that arguing with a
decreasing sequence of objective function values does not
suffice. The connection to strongly feasible power diagrams, however,
provides us with additional tools to prove termination.

\begin{theorem}\label{thm:termination}
  Algorithm \ref{algo:wbkm} terminates with a clustering that allows a
  strongly feasible centroidal power diagram.
\end{theorem}

\begin{pf}
  First, we prove that the center of gravity $c_i$ of a cluster is the unique
  optimal site with respect to a fixed weight-balanced least-squares
  assignment. Hence, if it terminates, the algorithm produces a
  clustering with a feasible centroidal power diagram.

  So, let $C:=(C_1,\dots,C_k)$ be a fixed clustering with centers of
  gravity $c_1,\dots,c_k$, and let $s_1,\dots,s_k$ be optimal
  sites. We set $z_i:=c_i-s_i$ for $i\leq k$. For each $C_i$, we then
  have
  \begin{eqnarray*}
   \sum\limits_{j=1}^n  y_{ij}\omega_j \|x_j-s_i\|^2  & = & \sum\limits_{j=1}^n y_{ij}\omega_j x_j^Tx_j - 2 \sum\limits_{j=1}^n y_{ij}\omega_j x_j^T(c_i-z_i) + \sum\limits_{j=1}^n y_{ij}\omega_j (c_i-z_i)^T(c_i-z_i)\\
    & = &\sum\limits_{j=1}^n y_{ij}\omega_j x_j^Tx_j -2|C_i|c_i^Tc_i + 2|C_i|c_i^Tz_i  + |C_i|c_i^Tc_i - 2|C_i|c_i^Tz_i  + |C_i|z_i^Tz_i\\
    & = & \sum\limits_{j=1}^n y_{ij}\omega_j x_j^Tx_j - |C_i|c_i^Tc_i + |C_i|z_i^Tz_i\;\;.
  \end{eqnarray*}
  Since the clustering is fixed, and since $z_i^Tz_i \geq 0$, this sum
  is minimal only for $z_i=0$, i.e., for $s_i=c_i$. Let 
$$
\Theta(C,S):=\sum\limits_{i=1}^k \sum\limits_{j=1}^{n}
  y_{ij}\omega_j\cdot (s_i^Ts_i-2x_j^Ts_i)
$$ 
be the least-squares value for the clustering $C$ and sites $S$, 
let $C^{(l)}$ be the returned optimal   clustering for the $l$-th iteration, 
and let $S^{(l)}$ be the set of sites used for the linear program to produce $C^{(l)}$. Note that
  $S^{(l+1)}$ consists of the centers of gravity of the clusters of
  $C^{(l)}$.

We have
$$
\Theta(C^{(l)},S^{(l)})\geq \Theta(C^{(l)}, S^{(l+1)}) \geq
  \Theta(C^{(l+1)}, S^{(l+1)})\;\;,
$$ 
hence the sequence $\Theta(C^{(l)},S^{(l)})$ is decreasing.The termination criterion in $(2.)$ is thus well-defined, and the above computation implies that the sequence is
  strictly decreasing until termination. By the final
  update of the sites as centers of gravity in $(2.)$, we return a
  feasible centroidal power diagram.

  In each iteration we compute a clustering corresponding to a vertex of the weight-balanced partition polytope. By Corollary \ref{cor:WeComputeSPD}, it allows a
  strongly feasible power diagram. The objective function values are strictly 
  decreasing, so no vertex is visited twice. Also, if a vertex stays optimal in the next step, we have reached a centroidal power diagram and the algorithm terminates. As the number of vertices is finite, this proves
  termination. \hfill $\Box$
\end{pf}

In the next section, we provide a worst-case upper bound on the number
of iterations of our algorithm.

\subsection{A bound on the number of iterations for strongly
  weight-balanced clustering}\label{subsec-iterations}

The first step towards a bound on the number of iterations of
Algorithm \ref{algo:wbkm} is to bound the number of different supports
of strongly feasible power diagrams for weight-balanced clusterings. 
This will be done by estimating the number of different point-cell incidence 
structures that can possibly be realized by power diagrams. 
To be more precise, for a power diagram $P:=(P_1,\ldots,P_k)$ let
$$
\mathcal{X}(P):=(X\cap P_1,\ldots, X\cap P_k)
$$ 
be the {\bf $(X,P)$-incidence pattern}. In general, i.e., if we do not want to particularly stress
a specific power diagram, we will speak of a {\bf power pattern}.  
Note that for a weight-balanced least-squares assignment $C$ and a corresponding
strongly feasible power diagram $P:=(P_1,\ldots,P_k)$ we have
$$
\text{supp(C)}=(\text{supp}(C_1),\dots,\text{supp}(C_k))=
(X\cap P_1,\ldots, X\cap P_k)=\mathcal{X}(P).
$$
Of course, since our definition of an $(X,P)$-incidence pattern
does not involve weight-balancing, the number of power patterns 
will in general be larger than the number of strongly feasible power diagrams 
for weight-balanced clusterings.

In order to provide an upper bound for the number of different power patterns, 
we use a well-known bound on the number of so-called sign-patterns of a set of
polynomials by Warren \cite{w-68}. (In \cite{w-68} only a bound
for $\{-1,1\}$-sign paterns is given; this can, however, easily be extended to
the $\{-1,0,1\}$-sign patterns that we need here; see e.g. \cite{alon-95}.)

Let $p_1,\dots,p_t$ be a system of real polynomials in $s$
variables. For a point $z\in \mathbb{R}^s$, the {\bf sign-pattern } of
$p_1,\dots,p_t$ is a tuple $v(z)=(v_1,\dots,v_t)^T\in \{+1,0,-1\}^{t}$
defined by $v_i=-1$ if $p_i(z)<0$, $v_i=0$ if $p_i(z)=0$ and $v_i=1$
if $p_i(z)>0$.

\begin{proposition}[ \cite{w-68}]\label{prop:poly} 
Let $p_1,\dots,p_t$ be a system of real polynomials in $s$  variables, 
all of degree at most $l$. If $s\leq 2t$, then the number of different 
sign-patterns of this system is bounded above by 
$$
\left(\frac{8e\cdot l\cdot t}{s}\right)^s\;\; .
$$
\end{proposition}

We will now use Proposition \ref{prop:poly} to give an upper bound on the 
number of different  power patterns.

 \begin{theorem}\label{thm:numberofsolutions}
Let $X:=\{x_1,\dots,x_n\}\subset\mathbb{R}^d$ be a (weighted) data set that has to be partitioned with respect to a (strongly) feasible power diagram with $k$ cells. Then there are at most 
$$
\left(\frac{8e\cdot (k-1)n}{d}\right)^{(d+1)k-1}
$$ 
different power patterns.
 \end{theorem}

\begin{pf}
  An $(S,\Sigma)$-power diagram with $k$ cells in $\mathbb{R}^d$ is
  defined by $k$ distinct sites $S:=\{s_1,\dots,s_k\}\subset
  \mathbb{R}^d$ and $k$ weights
  $\Sigma:=\{\sigma_1,\dots,\sigma_k\}\subset \mathbb{R}$. Its cells
  are the same as those of an $(S,\Sigma+\sigma')$-power diagram,
  where $\Sigma+\sigma'=\{\sigma_1+\sigma',\dots,\sigma_k+\sigma'\}$
  for some $\sigma'\in \mathbb{R}$. Due to this invariance we may choose
  $\sigma_k=0$, and it suffices to use a
  vector $$g=(s_1,\sigma_1,\dots,s_{k-1},\sigma_{k-1},s_k)^T \in
  \mathbb{R}^{(d+1)k-1}$$ to define the cells of the
  $(S,\Sigma)$-power diagram. We call these vectors {\em reduced
    control vectors}. 

For any pair $i\neq j$ and any point $x\in X$, we define a polynomial $$p_{i,j,x}(g):=(\|x-s_i\|^2-\sigma_i) - (\|x-s_j\|^2-\sigma_j).$$ The sign of $p_{i,j,x}(g)$ represents how $\|x-s_i\|^2-\sigma_i$ relates to $\|x-s_j\|^2-\sigma_j$ for the power diagram defined by $g$. Recall that $x$ lies in the $i$-th cell if and only if $\|x-s_i\|^2-\sigma_i \leq \|x-s_j\|^2-\sigma_j$ for all $j\neq i$. Note that there are $\binom{k}{2}\cdot n$ such polynomials and that they are quadratic in $s_i$ and $s_j$ (and linear in $\sigma_i, \sigma_j$).

Let $\mathcal{V}$ be the set of all sign patterns for the set of polynomials $p_{i,j,x}$ for all $i\neq j$ and $x\in X$. Each possible power pattern for $X$ corresponds to one or more sign patterns in $\mathcal{V}$. Hence, by bounding the number of polynomials in $\mathcal{V}$ from above, we also derive an upper bound on the number of power patterns for $X$.

We now apply Proposition \ref{prop:poly}. We have $\binom{k}{2}\cdot n$
polynomials, all of them are of degree $2$, and there are $(d+1)k-1$ variables. In our case, the condition $s \leq 2t$
translates to $(d+1)k-1\leq 2 \binom{k}{2}n$. It is satisfied since
from our general assumptions $d+1\leq n$ and $k\geq 2$ we obtain
$$(d+1)k-1 < (d+1)k \leq nk \leq (k-1)kn\;\;.$$

Hence, we obtain the upper bound 
$$
\left(\frac{8e\cdot 2  \binom{k}{2}n}{(d+1)k-1}\right)^{(d+1)k-1}
\leq \left(\frac{8e\cdot k(k-1)n}{(d+1)k-k}\right)^{(d+1)k-1}
= \left(\frac{8e\cdot (k-1)n}{d}\right)^{(d+1)k-1}\;\;
$$ 
which proves the assertion.  \hfill $\Box$
\end{pf}

Now we can prove our bound on the number of iterations for Algorithm \ref{algo:wbkm}.

\begin{theorem}\label{thm:numberofiterations}
The number of iterations of Algorithm \ref{algo:wbkm} is bounded by
$$
(40ek^2n)^{(d+1)k-1}.
$$
\end{theorem}

\begin{pf}
  By Corollary \ref{cor:WeComputeSPD}, in each iteration, the linear
  program of the algorithm computes a clustering that allows a
  strongly feasible power diagram.

  The sequence of objective function values of these linear programs
  is strictly decreasing. Thus, the number of iterations is bounded
  above by the number $|\mathcal{V}|$ of vertices of the weight-balanced partition
polytope corresponding to   clusterings that allow a strongly feasible power diagram.

  Given a clustering $C$ that allows a strongly feasible power
  diagram, we will bound the number of vertices of the weight-balanced
  partition polytope whose clusterings share the same power pattern. 
Combining this number with the bound in Theorem
  \ref{thm:numberofsolutions} then yields the claim.

  Let $y^*$ be the assignment vector of such a clustering $C$. By
  Proposition \ref{prop:LSAPD}, $y^*$ contains at most $2(k-1)$
  entries $y^*_{ij}$ with $0<y^*_{ij}<1$. This implies that at most $k-1$ points
  belong to more than one cluster. 

Let $X_{\mbox{\small{int}}}(y^*)$ denote the subset of $X$ of all points $x_j$ that belong to 
exactly one cluster. For these points, $y^*_{ij}=1$ for exactly one $i\leq k$, and $y^*_{ij}=0$ for all others.  In the following linear system of inequalities all the points of $X_{\mbox{\small{int}}}(y^*)$ are fixed by the third line.  

$$
\begin{array}{lcrclcl}
  \kappa_i^- & \leq   & \sum\limits_{j=1}^{n}  \omega_j\cdot y_{ij}    & \leq & \kappa_i^+ & \quad &(i\leq k) \hfill \;\\
  &      & \sum\limits_{i=1}^k y_{ij} & =  &  1
  & \quad &  (j\leq n)\hfill \;\\
  &      &          y_{ij}
  & =   & y^*_{ij}                   &  &  (i\leq k,j\leq n: y^*_{ij} \in \{0,1\})\hfill  \;\\
  &      &          y_{ij}
  & \geq   & 0                   &  & (i\leq k,j\leq n: y^*_{ij}\notin \{0,1\}) \hfill \;\;.
\end{array}
$$

This system encodes the {\bf restricted weight-balanced partition polytope $Q(y^*)$}
as the intersection of the weight-balanced partition polytope with the set of 
hyperplanes $y_{ij} = y^*_{ij}$ for all $y^*_{ij} \in \{0,1\}$. 
It represents all feasible clusterings $C':=(C_1',\ldots,C_k')$ in which all points of 
$X_{\mbox{\small{int}}}(y^*)$ are integrally assigned as specified. In particular, we have
$$
x_j\in X_{\mbox{\small{int}}}(y^*) \, \land \, \exists i,l: x_j\in \text{supp}(C'_i)\cap\text{supp}(C'_l)
\quad \Rightarrow \quad  i=l.
$$
Note that $y^*$, being a vertex of the weight-balanced partition polytope, is also a vertex 
of the polytope $Q(y^*)$.

Using the parameter vector $z:=(z_{11},\dots,z_{kn})^T\in
\mathbb{R}^{kn}$ defined by $z_{ij}:=1$ if $y_{ij}^*=1$, and $z_{ij}:=0$
otherwise, $Q(y^*)$ is equivalent to the polytope $Q'(y^*)$
$$
\begin{array}{lcrclcl}
  \kappa_i^- - \sum\limits_{j=1}^{n}  \omega_j\cdot z_{ij}& \leq   &  \sum\limits_{j:\, y^*_{ij}\notin \{0,1\}} \omega_jy_{ij}  & \leq & \kappa_i^+ - \sum\limits_{j=1}^{n}  \omega_j\cdot z_{ij} & \quad &(i\leq k) \hfill \;\\
  &      &   \sum\limits_{i:\, y^*_{ij}\notin \{0,1\}} y_{ij} & =  &  1 - \sum\limits_{i=1}^k z_{ij}
  & \quad & ( j \leq n: \exists y^*_{ij}\notin \{0,1\})\hfill \;\\
  &      &          y_{ij}
  & \geq   & 0                   &  & (i\leq k,j\leq n: y^*_{ij}\notin \{0,1\}) \hfill \;\;.
\end{array}
$$
As there are at most $2(k-1)$ values $0<y^*_{ij}<1$, there are at most
$2(k-1)$ variables, and thus $Q'(y^*)$ has dimension at most
$\mathbb{R}^{2(k-1)}$. Further, $Q'(y^*)$ is presented by at most $5k-2$
constraints: line $1$ of the above formulation contributes at most $2k$ constraints, line $2$
at most $k-1$, and line $3$ at most $2(k-1)$.

The number of vertices of the weight-balanced partition polytope that
correspond to such clusterings $C'$ and that allow a strongly feasible 
power diagram is bounded above by the number of vertices of $Q'(y^*)$. 
This number does not exceed the
number of bases for the system of inequalities defining $Q'(y^*)$. The
latter is at most $\binom{5k-2}{2(k-1)}\leq (5k)^{2k-2}$.

By Theorem \ref{thm:numberofsolutions}, the number of power patterns
is at most $\frac{8e\cdot   (k-1)n}{d}^{(d+1)k-1}$. 
As pointed out before, this is also an upper bound for the number of supports of strongly feasible 
power diagrams. For each of these supports, the restricted weight-balanced partition polytope has 
at most $(5k)^{2k-2}$ vertices. 
Since, during a run of Algorithm \ref{algo:wbkm}, no vertex is visited twice, we obtain the
upper bound 
$$
\begin{aligned}
(5k)^{2k-2}\cdot \left(\frac{8e\cdot (k-1)n}{d}\right)^{(d+1)k-1} 
& \leq \left(\frac{5k\cdot 8e\cdot (k-1)n}{d}\right)^{(d+1)k-1}
&\le (40ek^2n)^{(d+1)k-1}
\end{aligned}
$$
on the number of its iterations. \hfill $\Box$
\end{pf}

As the proof of Theorem \ref{thm:numberofiterations} shows, the number of
different strongly feasible power diagrams is bounded by the same number.
Hence, with the aid of Proposition \ref{prop:LSAPD} we obtain the following 
corollary.

\begin{corollary}\label{cor:number}
For given $X$, $\Omega$ and $\kappa^{\pm}$, the number of strict weight-balanced 
least-squares assigments is bounded by
$$
(40ek^2n)^{(d+1)k-1}.
$$
\end{corollary}

\vspace*{-0.4cm}

\section{Final Remarks}\label{sec-final}

Note that in each iteration of Algorithm \ref{algo:wbkm} the feasibility region of 
the linear program in Step $(1.)$ is the same. An optimal assignment $y^*$ stays therefore 
feasible for the subsequent iteration and can hence be used
for a warm start. If the centers of gravity did not change significantly, the new optimal
assignment $y^{**}$ will be only a few primal simplex pivot steps away from $y^*$. 

Further, we would like to point out that it is particularly simple to kernelize Algorithm \ref{algo:wbkm}, which is a necessity in many applications.  For some background about kernel functions, see e.g. \cite{ss-02}. The key idea is the use of a kernel function $\Phi:\mathbb{R}^d\times \mathbb{R}^d \rightarrow \mathbb{R}$ in place of the standard inner product for points in the original data space $\mathbb{R}^d$. The only inner products that appear in the algorithm are in the objective function of the underlying linear program, and are of the type $x_j^Ts_i$ or
$s_i^Ts_i$. By preprocessing the $k\cdot n + k\cdot k=k(n+k)$ values
$\Phi(x_j,s_i)$ and $\Phi(s_i,s_i)$, we obtain a kernelized version of the algorithm.

\vspace*{-0.2cm}
\bibliographystyle{plain}

 \bibliography{sample}

\end{document}